\documentclass[a4paper,english,11pt]{amsart}
\usepackage{babel}
\usepackage{units}
\usepackage{amstext,amsmath}
\usepackage{amssymb}
\usepackage{amsthm}
\usepackage{graphicx}
\usepackage{color}
\usepackage{dsfont}
\usepackage{subfig}
\usepackage{tikz}
\usepackage{enumitem}
\usepackage{todonotes}
\usepackage{booktabs}
\usepackage{mathtools}
\usepackage{forloop}

\usepackage[left = 2.5cm, right = 2.5cm, top= 2cm]{geometry}
\definecolor{hanblue}{rgb}{0.27, 0.42, 0.81}
\definecolor{red}{rgb}{1.0, 0.0, 0.0}
\usepackage[colorlinks, citecolor=hanblue]{hyperref}
\usepackage{cleveref}

\makeatletter
\newtheorem{thm}{Theorem}
\newtheorem*{thmnonumber}{Theorem}
\theoremstyle{plain}
\theoremstyle{remark}
\newtheorem{rem}[thm]{Remark}
\theoremstyle{plain}
\newtheorem{lem}[thm]{Lemma}
\theoremstyle{plain}

\theoremstyle{plain}
\newtheorem{cor}[thm]{Corollary}
\theoremstyle{definition}
\newtheorem{definition}[thm]{Definition}

\theoremstyle{plain}

\theoremstyle{plain}

	\newcommand{\bes}{\begin{equation*}}
	\newcommand{\ees}{\end{equation*}}
	\newcommand{\be}{\begin{equation}}
	\newcommand{\ee}{\end{equation}}
	\newcommand{\bi}{\begin{itemize}}
	\newcommand{\ei}{\end{itemize}}
		\def\bas#1\eas{\begin{align*}#1\end{align*}}
		\def\ba#1\ea{\begin{align}#1\end{align}}
	\newcommand{\baed}{\begin{aligned}}
	\newcommand{\eaed}{\end{aligned}}
	\newcommand{\R}{\mathbb{R}}
	\newcommand{\N}{\mathbb{N}}

		\newcommand{\Om}{\Omega}

	\newcommand{\M}{\mathcal{M}}
	\newcommand{\calM}{\mathcal{M}}

	\newcommand{\calD}{\mathcal{D}}

  \let\div\relax
  \DeclareMathOperator*{\div}{div}

	\newcommand{\norm}[1]{\left\lVert#1\right\rVert}

	\newcommand{\p}{\partial}
	\newcommand{\de}{\partial}

	\newcommand{\olom}{\overline{\Om}}
        \newcommand{\WT}{X}
        
\newcommand{\f}{\varphi}
     
\newcommand{\zak}{%
  \mathbin{\vrule height 1.6ex depth 0pt width
0.13ex\vrule height 0.13ex depth 0pt width 1.3ex}
}    %

\DeclareMathOperator{\supp}{supp} %
\newcommand{\curves}{C_{\rm w} ([0,1];\M(\overline{\Om}))}
\newcommand{\pcurves}{C_{\rm w} ([0,1];\M^+(\overline{\Om}))}
\newcommand{\weakstar}{\stackrel{*}{\rightharpoonup}}  %

\DeclareMathOperator{\ext}{Ext} %
\DeclareMathOperator{\gr}{graph}

\newcommand{\AC}{{\rm AC}}

 \newcommand{\points}{\mathcal{C}_{\alpha,\beta}}
 \newcommand{\prob}{\mathcal{P}}

\author{Kristian Bredies}
\address[Kristian Bredies]{Institute of Mathematics and Scientific Computing, University of Graz, Heinrichstra\ss{}e 36, 8010 Graz, Austria.} 
\email{kristian.bredies@uni-graz.at}

\author{Marcello Carioni}
 \address[Marcello Carioni]{University of Cambridge, Department of Applied Mathematics and Theoretical Physics, Wilberforce Road, Cambridge
CB3 0WA, UK}
\email{mc2250@maths.cam.ac.uk}

\author{Silvio Fanzon}
\address[Silvio Fanzon]{Institute of Mathematics and Scientific Computing, University of Graz, Heinrichstra\ss{}e 36, 8010 Graz, Austria.}
\email{silvio.fanzon@uni-graz.at}

\author{Francisco Romero}
\address[Francisco Romero]{Institute of Mathematics and Scientific Computing, University of Graz, Heinrichstra\ss{}e 36, 8010 Graz, Austria.} 
\email{francisco.romero-hinrichsen@uni-graz.at}

\title{On the extremal points of the ball of the Benamou--Brenier energy}

\begin{document}
\maketitle

\begin{abstract}
\sloppy
In this paper we characterize the extremal points of the unit ball of the Benamou--Brenier energy and of a coercive generalization of it, both subjected to the homogeneous continuity equation constraint. We prove that extremal points consist of pairs of measures concentrated on absolutely continuous curves which are characteristics of the continuity equation. Then, we apply this result to provide a representation formula for sparse solutions of dynamic inverse problems with finite dimensional data and optimal-transport based regularization.
\end{abstract}

\medskip

\begin{center}

\vspace*{1mm}

\end{center}

\section{Introduction}

The classical theory of Optimal Transport deals with the problem of efficiently transporting mass from a probability distribution into a target one. In the last thirty years, great advances in the understanding of the underlying theory have been achieved \cite{ags, cuturi,sant}. However, only recently these techniques are starting to be applied in order to solve computational problems in a great variety of fields, with logistic problems \cite{bernot, mass, cities, equilibrium}, crowd dynamics \cite{macroscopic, handlingcongestion}, image processing \cite{hmp,mrss,mbmoj,histogram,papadakis, activecontours,colourmapping}, inverse problems \cite{bf,kr} and machine learning \cite{wassersteingan, wassersteinloss, learninggenerative, wassersteintraining, sae, wae} being a few examples.

In this paper we focus on the so-called Benamou--Brenier formula, 
which provides an equivalent dynamic formulation of the classical Monge--Kantorovich transport problem \cite{kantorovich}. 
Introduced by Benamou and Brenier in \cite{bb}, such formula allows to compute an optimal transport between two probability measures $\rho_0$ and $\rho_1$ on a closed bounded domain $\olom \subset \R^d$ through the minimization of the kinetic energy 
\begin{equation} \label{intro:bb}
\frac{1}{2}\int_0^1\int_{\olom} |v_t(x)|^2 d\rho_t(x)\,,
\end{equation} 
among all the pairs $(\rho_t,v_t)$, where $t \mapsto \rho_t$ is a curve of probability measures on $\olom$, $v_t \colon \olom \to \R^d$ is a time-dependent vector field and the pair $(\rho_t,v_t)$ satisfies distributionally the continuity equation
\begin{equation} \label{intro:cont}
\partial_t \rho_t + \div(\rho_t v_t) = 0 \quad \text{ subjected to } \quad \rho_{t=0}=\rho_0 , \,\, \rho_{t=1}=\rho_1 \,.
\end{equation}
The interest around the Benamou--Brenier formulation is motivated by its remarkable properties.  First, it allows to compute an optimal transport in an efficient way \cite{bb} by means of a convex reformulation of \eqref{intro:bb}, by introducing the momentum $m_t=\rho_tv_t$. More precisely, denoting by  $X:=(0,1) \times \olom$ the time-space cylinder, the Benamou--Brenier energy \eqref{intro:bb} can be equivalently defined as a convex functional on the space of bounded Borel measures $\M:=\M(X) \times \M(X;\R^d)$ by setting
\begin{equation} \label{intro:BB_convex}
B (\rho,m):= \frac{1}{2}\int_{X} \left| \frac{dm}{d\rho}(t,x)\right|^2 \, d \rho(t,x) \,,
\end{equation}
whenever $(\rho,m)\in \M$ are such that $\rho \geq 0$, $m \ll \rho$, and $B:=+\infty$ otherwise. 
With this change of variables, the continuity equation at \eqref{intro:cont} assumes the form
\begin{equation}\label{eq:contmeasures}
\partial_t \rho + \div m = 0 \,.
\end{equation}
In addition, the  dynamic nature of the Benamou--Brenier reformulation of optimal transport is at the core of many recent developments in the fields of PDEs, optimal transport and inverse problems. Indeed, the dynamic formulation allows to endow the space of probability measures with a differentiable structure \cite{ags,sant}, making possible the characterization of differential equations as gradient flows in spaces of measures \cite{agsinventiones,ottokinderleher2,otto1} or the derivation of sharp inequalities \cite{cordero,maggi,otto_villani}.  Moreover, it motivated recent developments in unbalanced optimal transport theory \cite{chizat2,chizat,kmv,liero}, i.e., when the marginals are arbitrary positive measures. Finally, as the Benamou--Brenier energy provides a description of the optimal flow of the transported mass at each time $t$, which is a valuable information in applications, it was recently employed as a regularizer for variational inverse problems \cite{inprep,inprep2,bf, hmp,mrss,schmitzerwirth2020}.

\smallskip

The goal of this paper is to characterize the extremal points of the unit ball of the Benamou--Brenier energy $B$ at \eqref{intro:BB_convex}, and of a coercive version of it, which is obtained by adding the total variation of $\rho$ to $B$. Both functionals are constrained via the continuity equation \eqref{intro:cont}.  Precisely, we introduce the functional 
\begin{equation}\label{intro:J}
J_{\alpha,\beta}(\rho,m):=   \beta B(\rho,m) + \alpha \|\rho\|_{\mathcal{M}(X)}  \quad \text{ subjected to } \quad \de_t \rho+ \div m=0 \,,
\end{equation}
defined for all $(\rho,m) \in \M$ and $\alpha \geq 0$, $\beta>0$. We then characterize the extremal points of the subset of $\M$ defined by
\[
C_{\alpha,\beta}:=\{(\rho,m) \in \M \, \colon \, J_{\alpha,\beta} (\rho,m)  \leq 1\}\,.
\]
We emphasize that we do not enforce initial conditions to the continuity equation in \eqref{intro:J}. To be more specific, we prove the following result (see Theorem \ref{thm:extremal}).
\begin{thmnonumber} 
Let $\alpha \geq 0$, $\beta>0$. The extremal points of the set $C_{\alpha,\beta}$  
  are exactly given by the zero measure $(0,0)$ and the pairs of measures $(\rho,m)$ such that 
  \[
  \rho=a_{\gamma} \, dt \otimes  \delta_{\gamma(t)},\, \,\,\, \, m=\dot{\gamma}(t) a_\gamma\, dt \otimes  \delta_{\gamma(t)} , \,\,\,\,\, a_\gamma = \left(\frac\beta2 \int_0^1 |\dot\gamma(t)|^2 \, dt + \alpha\right)^{-1}\,,
  \]
  where $\gamma: [0,1] \to \olom$ is an absolutely continuous curve with weak derivative $\dot\gamma \in L^2$,
  and such that $a_{\gamma}<+\infty$. If $\alpha = 0$ the condition $a_{\gamma}<+\infty$ is satisfied if and only if $\gamma$ is not constant. 
\end{thmnonumber}
We therefore show that the extremal points of the set $C_{\alpha,\beta}$ are pairs of measures $(\rho,m)$, with $\rho$ concentrated on some absolutely continuous curve $\gamma$ in $\olom$, and the density of $m$ with respect to $\rho$ is given by $\dot \gamma$ . Notice that such conditions are equivalent to the existence of a measurable field $v \colon X \to \R^d$ such that 
\begin{equation}\label{ode}
\dot{\gamma}(t)=v(t,\gamma(t)) \quad \mbox{ for a.e.} \quad t \in (0,1)\,,
\end{equation} 
thus showing that $\gamma$ is a characteristic associated to the continuity equation at \eqref{intro:cont} with respect to the field $v$.
We prove the above Theorem in Section \ref{sec:characterizationextremal}, with the aid of a probabilistic version of the superposition principle for positive measure solutions to the continuity equation \eqref{intro:cont} on the domain $(0,1) \times \olom$ (see Theorem \ref{thm:lifting}). We mention that the ideas behind such superposition principle are not new, and they were  originally introduced in \cite{ambrosioinventiones} for positive measures on $(0,1) \times \R^d$ (see also \cite{ags, bernard, trevisan}). The result of Theorem \ref{thm:lifting} allows to decompose any measure solution $(\rho,m)$ of the continuity equation \eqref{eq:contmeasures} with bounded Benamou--Brenier energy, as superposition of measures concentrated on absolutely continuous characteristics of \eqref{eq:contmeasures}, that is, curves solving \eqref{ode} with $v=dm/d\rho$. As a consequence, we show any pair of measures that is not of such a form can be written as a proper convex combination of elements of $C_{\alpha,\beta}$ and thus it is not an extremal point.  The opposite inclusion follows from the convexity of the energy and the properties of the continuity equation.

\smallskip

The interest in characterizing extremal points of the Benamou--Brenier energy is not only theoretical. It has been recently shown in \cite{bc} and \cite{chambolle} that in the context of variational inverse problems with finite-dimensional data, the structure of sparse solutions is linked to the extremal points of the unit ball of the regularizer. 
In the classical theory of variational inverse problems one aims to solve
\begin{equation}\label{eq:inverseproblemintro}
\min_{u\in \mathcal{U}} \ R(u) \quad \text{ subjected to } \quad  Au = y\,,
\end{equation} 
where $\mathcal{U}$ is the target space, $R$ is a convex regularizer, $A$ is a linear observation operator mapping to a finite-dimensional space and $y$ is the observation.
It has been empirically observed that the presence of the regularizer $R$ is promoting the existence of sparse solutions, namely minimizers that can be represented as a finite linear combination of simpler atoms. 
While this effect has been well-understood in the case when $\mathcal{U}$ is finite dimensional, the infinite-dimensional case has been only recently addressed \cite{chambolle, bc, sparsetik, flinthweiss, unser3, unser2, unsersplines}. In particular, in \cite{chambolle,bc}, it has been shown that, under suitable assumptions on $R$ and $A$, there exists a minimizer of \eqref{eq:inverseproblemintro} that can be represented as a finite linear combination of extremal points of the unit ball of $R$; namely the atoms forming a sparse solution are the extremal points of the ball of the regularizer.

In view of the above discussion, in Section \ref{sec:application} we apply our characterization of the extremal points of the energy $J_{\alpha,\beta}$ at \eqref{intro:J} to understand the structure of sparse solutions for inverse problems with such transport energy acting as regularizer. We mention that the analysis is carried out for the case $\alpha > 0$, as the functional $J_{0,\beta}$, corresponding to the rescaled Benamou--Brenier energy, lacks of compactness properties (see Remark \ref{rem:compact}). We verify that the assumptions needed to apply the representation theorems in \cite{bc} and \cite{chambolle} are satisfied by $J_{\alpha,\beta}$, and consequently we deduce the existence of a minimizer that is given by a finite linear combination of measures concentrated on absolutely continuous curves in $\olom$ (see Theorem \ref{thm:sparserepresentation}). As a specific application of Theorem \ref{thm:sparserepresentation} we consider the setting introduced in \cite{bf}, where the regularizer $J_{\alpha,\beta}$ is coupled with a fidelity term that penalizes the distance between the unknown measure $\rho_t$ computed at $t_1,\ldots,t_N \in (0,1)$, and the observation at such times (see Section \ref{subsec:dynamic}). This setting is relevant for applications, such as variational reconstruction in undersampled dynamic MRI. Employing the previous results we are able to prove the existence of a sparse solution represented with a finite linear combination of measures concentrated on absolutely continuous curves in $\olom$ (see Corollary \ref{corollary}).

\smallskip

To conclude, we mention that characterizing the extremal points for a given regularizer has important consequences in devising algorithms able to compute a sparse solution. Notable examples have been proposed for the total variation regularizer in the space of measures \cite{boyd, brediesp} using so-called \emph{generalized conditional gradient methods} (or Frank--Wolfe-type algorithms \cite{frankwolfe}). Inspired by the previous methods, and building on the theoretical results obtained in the present paper, we plan to develop numerical algorithms to compute sparse solutions of dynamic inverse problems with the optimal transport energy $J_{\alpha,\beta}$ as a regularizer \cite{inprep,inprep2}, effectively providing a numerical counterpart to the theoretical framework established in \cite{bf}. 
Finally, we remark that similar results to the ones presented in this paper can be obtained for unbalanced optimal transport energies. This has been recently achieved in \cite{superpositioninhomogeneous}, by introducing a novel superposition principle for measure solutions to the inhomogeneous continuity equation.

\section{Mathematical setting and preliminaries}

In this section we give the basic notions about the continuity equation, the Benamou--Brenier energy, and its coercive version $J_{\alpha,\beta}$  anticipated in the introduction. We refer to \cite{ ags, bb, bf} for a more detailed overview. For measure theoretical notions, we refer to the definitions in \cite{afp}.

Given a metric space $Y$ we will denote by $\M(Y)$ (resp.~$\M(Y ; \R^d)$)  the space of bounded Borel measures (resp.~bounded vector Borel measures) on $Y$. Similarly,  $\M^+(Y)$ and $\mathcal{P}(Y)$ denote the set of bounded positive Borel measures and Borel probability measures on $Y$, respectively.
Let \(\Omega \subset \R^d\) be an open, bounded domain with \(d\in \N, d \geq 1\). 
Set \(\WT := (0,1) \times \overline \Omega\), 
\[
  \M := \M(X) \times \M( X ; \R^d)\, ,
\] 
and
\[
  \calD := \{(\rho,m) \in \calM \ : \ \p_t \rho + \div m = 0 \ \ \text{in} \ \ 
  \WT \}\,,
\]
where the solutions of the continuity equation are intended in a distributional 
sense, that is,
\begin{equation} \label{cont weak}
\int_{X}  \de_t \f \, d\rho + 
\int_{X}  \nabla \f \cdot  dm   = 0   \quad \text{for all} \quad \f \in C^{\infty}_c ( X ) \,.
\end{equation}

We remark that the above weak formulation includes no-flux boundary conditions for the momentum $m$ on $\de \Om$. Also, no initial and final data is prescribed in \eqref{cont weak}. Moreover, by standard approximation arguments, we can consider in \eqref{cont weak} test functions in $C^1_c ( X )$ (see \cite[Remark 8.1.1]{ags}).

We now introduce the Benamou--Brenier energy.
For this purpose, define the convex, lower semicontinuous and one-homogeneous map $\Psi \colon \R \times \R^d \to [0,\infty]$ by setting
\[
\Psi(t,x):=
\begin{cases}
\frac{|x|^2}{2t} & \text{ if } t>0\,, \\
0    & \text{ if } t=|x|=0 \,, \\
+ \infty& \text{ otherwise}\,. 
\end{cases}
\]
The Benamou--Brenier energy $B: \M \to [0,\infty]$ is defined 
for every pair $(\rho,m) \in \mathcal{M}$ as
\begin{equation} \label{intro convex}
B (\rho,m):= \int_X \Psi \left( \frac{d\rho}{d\lambda}, \frac{dm}{d\lambda}\right) \, d \lambda \,,
\end{equation}
$\lambda \in \M^+(X)$ is such that $\rho,m \ll \lambda$. 
Since $\Psi$ is one-homogeneous, the above representation of $B$ does not depend on $\lambda$. 
For some
fixed $\alpha \geq 0$, $\beta > 0$, we consider the following functional
\begin{equation} \label{prel:reg}
  J_{\alpha, \beta}(\rho,m) := \begin{cases} 
      \beta B(\rho, m) +
      \alpha \norm{\rho}_{\M(X)} & \,\, \text{ if }(\rho,m) \in \calD, \\ 
      + \infty\qquad  & \,\,   \text{ otherwise},
    \end{cases} 
\end{equation}
where $\norm{\cdot}_{\M(X)}$ denotes the total variation norm in  \(\calM(\WT)\).  
\begin{rem} \label{rem:compact}
Note that in the definition of $J_{\alpha,\beta}$ we add the total variation of $\rho$ to the Benamou--Brenier energy. If $\alpha > 0$ this choice enforces the balls of the energy $J_{\alpha, \beta}$ to be compact in the weak* topology of $\M$ (see Lemma \ref{lem:prop J}). As a consequence, the functional $J_{\alpha,\beta}$  is a natural regularizer for dynamic inverse problems when the initial and final data are not prescribed \cite{bf}.
We remark that, although in the case $\alpha = 0$ the unit ball of the energy $J_{0, \beta}$ is not compact, we can still characterize its extremal points. However, in this case, due to the lack of coercivity, $J_{0,\beta}$ has limited use as a regularizer for dynamic inverse problems.
\end{rem}

For a measure $\rho \in \M(X)$, we say that $\rho$ disintegrates with respect to time if there exists a Borel family of measures $\{\rho_t\}_{t \in [0,1]}$ in %
$\M(\olom)$ such that
	\[
	\int_X \f(t,x) \, d\rho(t,x) = \int_0^1 \int_{\olom} \f (t,x) \, d\rho_t(x) \, dt \quad \text{ for all } \quad  \f \in L^1_{\rho}(X) \,,
	\]
We denote such disintegration with the symbol $\rho =dt \otimes  \rho_t$. Further, we say that a curve of measures $t \in [0,1] \mapsto \rho_t \in \M(\olom)$ is narrowly continuous if the map
\[
t \mapsto \int_{\olom} \f(x) \, d\rho_t(x)
\]
is continuous for each fixed $\f \in C(\olom)$. The family of narrowly continuous curves will be denoted by $\curves$. We also introduce $\pcurves$, as the family of narrowly continuous curves with values into the positive measures on $\olom$.

We now recall several results about $B$, $J_{\alpha,\beta}$ and measure solutions of the continuity equation \eqref{cont weak}, which will be useful in the following analysis. For proofs of such results, we refer the interested reader to Propositions 2.2, 2.4, 2.6 and Lemmas 4.5, 4.6 in \cite{bf}.

\begin{lem}[Properties of $B$] \label{lem:prop B}
	The functional $B$ defined in \eqref{intro convex} is convex, positively one-homogen\-eous and sequentially lower semicontinuous with respect to the weak* topology on $\M$. Moreover it satisfies the following properties:
	\begin{enumerate}
	\item[i)] $B(\rho,m) \geq 0$ for all $(\rho,m)\in \M$,
	\item [ii)] if $B(\rho,m)<+\infty$, then $\rho \geq 0$ and $m \ll \rho$, that is, there exists a measurable map $v \colon X \to \R^d$ such that $m=v\rho$,
	\item [iii)] if $\rho \geq 0$ and $m = v \rho$ for some $v \colon X \to \R^d$ measurable, then
\begin{equation} \label{formula B}
B(\rho,m) =\int_X \Psi (1,v) \, d\rho =  \frac12 \int_X |v|^2 \, d\rho \,. 
\end{equation}
	\end{enumerate}
\end{lem}

\begin{lem}[Properties of the continuity equation] \label{lem:prop cont}
	Assume that $(\rho,m) \in \M$ satisfies \eqref{cont weak} and that $\rho \in \M^+(X)$. Then $\rho$ disintegrates with respect to time into $\rho =dt \otimes  \rho_t$, 
	where $\rho_t \in \M^+(\overline{\Om})$ for a.e.~$t$. Moreover $t \mapsto \rho_t(\overline{\Om})$ is constant, with $\rho_t(\overline{\Om}) = \rho(X)$ for a.e. $t \in (0,1)$. If in addition $B(\rho,m)< +\infty$, that is, 
	\[
	\int_0^1\int_{\olom} |v|^2 \, d\rho_t(x) \, dt < +\infty\,,
	\]
where $m=v\rho$ for some $v \colon X \to \R^d$ measurable, 
then $t \mapsto \rho_t$ belongs to $\pcurves$. 
\end{lem}

\begin{lem}[Properties of $J_{\alpha,\beta}$] \label{lem:prop J}
Let $\alpha \geq 0$, $\beta >0$. The functional $J_{\alpha,\beta}$ is non-negative, convex, positively one-homogeneous and sequentially lower semicontinuous with respect to weak* convergence on $\M$. Assume now $\alpha>0$.
		For $(\rho,m) \in \mathcal{M}$ such that $J_{\alpha,\beta}(\rho,m)< +\infty$ we have  
	\begin{equation} \label{lem:prop J est}
          \alpha \norm{\rho}_{\M(X)} \leq J_{\alpha,\beta}(\rho,m)
          \,, \quad \min(2\alpha,\beta) \norm{m}_{\M(X;\R^d)} \leq
          J_{\alpha,\beta}(\rho,m) \,.
	\end{equation}	
	Moreover, if $\{(\rho^n,m^n)\}_n$ is a sequence in $\M$ such that
\[
	\sup_n \ J_{\alpha,\beta} (\rho^n,m^n) < + \infty\,,
\]
then $\rho^n = dt \otimes \rho_t^n$ for some $(t \mapsto \rho_t^n) \in \pcurves$ and there exists some
	$(\rho,m) \in \mathcal{D}$ with $\rho= dt \otimes \rho_t$, $\rho_t \in \pcurves$ such that, up to subsequences, 
	\begin{equation} \label{topology} \left\{
	\begin{gathered}
	(\rho^n,m^n) \weakstar (\rho,m) \,\, \text{ weakly* in } \,\, \M \,, \\
	\rho_t^n \weakstar \rho_t \,\, \text{ weakly* in }  \,\, \M(\olom)\,,\, \text{ for every } \,\, t \in [0,1]\,.   
	\end{gathered} \right.
	\end{equation}
\end{lem}

\section{Characterization of extremal points}\label{sec:characterizationextremal}

The aim of this section is to characterize the extremal points of the unit ball of the functional $J_{\alpha,\beta}$ at \eqref{prel:reg} for all $\alpha \geq 0, \beta>0$, namely, of the convex set
\[
C_{\alpha,\beta} := \left\{(\rho,m) \in \M \, \colon \, J_{\alpha,\beta}(\rho,m)  \leq 1 \right\}\,. 
\]
To this end, let us first introduce the following set.
\begin{definition}[Characteristics] \label{def:char}
For $\alpha\geq 0$, $\beta > 0$ define the set $\points$ of all the pairs $(\rho,m) \in \mathcal{M}$ such that
\[
\rho=a_\gamma \, dt \otimes  \delta_{\gamma(t)}\,, \,\,\,\,\, m=\dot{\gamma}(t)\, a_\gamma \, dt \otimes  \delta_{\gamma(t)} \,, \,\,\,\,\,
a_{\gamma} := \left( \frac{\beta}{2} \int_0^1 |\dot\gamma(t)|^2 \, dt + \alpha \right)^{-1}\,,
\] 
where $\gamma \in \AC^2([0,1];\R^d)$ satisfies $\gamma(t) \in \olom$ for each $t \in [0,1]$ and $a_{\gamma}<+\infty$.  
\end{definition} 
We remind that $\AC^2([0,1];\R^d)$ denotes the space of absolutely continuous curves having a weak derivative in $L^2$.
We point out that by definition $a_{\gamma} >0$ for all choices of $\alpha\geq 0$, $\beta>0$. Moreover the condition $a_{\gamma}<+\infty$ is always satisfied if $\alpha>0$. When $\alpha=0$ we instead have $a_{\gamma}<+\infty$ if and only if $\int_0^1 |\dot{\gamma}(t)|^2\, dt > 0$, that is, the set $\mathcal{C}_{0,\beta}$ does not contain constant curves.

For the extremal points of $C_{\alpha,\beta}$ we have the following characterization.
\begin{thm} \label{thm:extremal}
Let $\alpha\geq 0$, $\beta>0$ be fixed. Then
	\[
	\ext(C_{\alpha,\beta})=\{(0,0) \} \cup \points \,.
	\]
\end{thm}

The proof of Theorem \ref{thm:extremal} is postponed to Section \ref{subsection:proof}. In order to show the inclusion  
$\ext(C_{\alpha,\beta}) \subset \{(0,0)\} \cup \points$ we will make 
use of a superposition principle for measure solutions of the continuity equation \eqref{cont weak}. This result is not new, and it is proved in \cite[Ch 8.2]{ags} for the case $\Om = \R^d$. In Section \ref{subsec:superposition}  we show that it also holds for bounded closed domains.

\subsection{The superposition principle} \label{subsec:superposition}

 Before stating the superposition principle in $\olom$, we introduce the following notation.  
Let
\[
\Gamma := \left\{ \gamma \colon [0,1] \to \R^d  \, \colon \, \gamma \text{ continuous} %
\right\}
\]
be equipped with the supremum norm, i.e., $\norm{\gamma}_{\infty}:=\max_{t \in [0,1]} |\gamma(t)|$. For every fixed $t \in [0,1]$ let $e_t \colon \Gamma \to \R^d$ be the evaluation at $t$, that is,  $e_t(\gamma):=\gamma(t)$. Notice that $e_t$ is continuous.  For a measurable vector field $v \colon (0,1) \times \R^d \to \R^d$, we define the following subset of $\Gamma$ consisting of $\AC^2$ curves solving the ODE \eqref{ode} in the sense of Carath\'eodory:
\[
\Gamma_{v}(\R^d) := \left\{ \gamma \in \Gamma  \, \colon \, \gamma \in \AC^2([0,1];\R^d),\,\,  \dot{\gamma}(t)=v(t,\gamma(t)) \text{ for a.e. } t \in (0,1) \right\}\,.
\]
Moreover define the set of solutions to the ODE which live inside $\olom$ for all times:
\[
\Gamma_{v}(\olom) := \left\{ \gamma \in \Gamma_v (\R^d)\, \colon \, \gamma(t) \in \olom \,\, \text{ for all } \,\, t \in [0,1]
\right\}  \,.
\]
The superposition principle for probability solutions to \eqref{cont weak} states as follows.

\begin{thm} \label{thm:lifting}
Let $t \in [0,1] \mapsto \rho_t \in \prob (\olom)$ be a narrowly continuous solution of the continuity equation in the sense of \eqref{cont weak}, for some measurable $v \colon (0,1) \times \olom \to \R^d$ such that 
\begin{equation} \label{thm:lifting:1}
\int_0^1 \int_{\olom} |v(t,x)|^2 \, d\rho_t (x) \, dt < + \infty\,.
\end{equation}
Then there exists a probability measure $\sigma \in \prob(\Gamma)$ concentrated on $\Gamma_{v}( \olom)$ and such that $\rho_t = (e_t)_\# \sigma$ for every $t \in [0,1]$, that is,
\begin{equation} \label{representation:1}
\int_{\olom} \f (x) \, d\rho_t (x) = \int_\Gamma \f (\gamma(t)) \, d\sigma(\gamma) \quad \text{ for every } \,\, \f \in C(\olom),\,\, t \in [0,1]\,. 
\end{equation}
\end{thm}

\begin{proof}
Let $\bar v :(0,1) \times \R^d \rightarrow \R^d$ be the extension to zero of $v$ to the whole $\R^d$.
Similarly, for each $t \in [0,1]$, let $\bar{\rho}_t \in \prob(\R^d)$ be the extension to zero of $\rho_t$ in $\R^d$. Note that the pair $(\bar \rho, \bar v \, \bar \rho)$ is a solution of the continuity equation in $(0,1) \times \R^d$ in the sense of \eqref{cont weak}. Moreover $\bar{\rho}$ and $\bar{v}$ satisfy \eqref{thm:lifting:1} in $(0,1) \times \R^d$. Therefore we can apply Theorem 8.2.1 in \cite{ags} and obtain a probability measure $\sigma \in \prob (\Gamma)$ concentrated on $\Gamma_{\bar{v}} (\R^d)$ and such that $\bar{\rho}_t = (e_t)_\# \sigma$ for all $t \in [0,1]$, that is,
\begin{equation} \label{representation:2}
\int_{\R^d} \f (x) \, d\bar{\rho}_t (x) = \int_\Gamma \f (\gamma(t)) \, d\sigma(\gamma) \quad \text{ for every } \,\, \f \in C_b(\R^d),\,\, t \in [0,1]\,.
\end{equation}
We claim that $\sigma$ is concentrated on $\Gamma_{v}(\olom)$. In order to show that, partition $\Gamma_{\bar{v}} (\R^d)$ into 
\[
\Gamma_{\bar{v}}(\R^d) = \Gamma_{\bar{v}}(\olom) \cup A \,,
\]
where
\[
A:= \left\{ \gamma \in \Gamma_{\bar{v}}(\R^d) \, \colon \, \text{there exists } \,\, \hat{t} \in [0,1] \,\, \text{such that} \,\, \gamma(\hat{t}) \in \olom^c \right\} \,.
\]
Notice that, since $\olom^c$ is open and $v \equiv  0$ in $\olom^c$, the curves in $A$ are constant, so that we can write
\[
A = \left\{ \gamma \in \Gamma_{\bar{v}}(\R^d) \, \colon \,  \gamma(0) \in \olom^c \right\} \,.
\]
From this, it follows that $A \subset e_0^{-1}(\olom^c)$. Moreover, \eqref{representation:2} implies $\bar{\rho}_0 (\olom^c) = \sigma(e_0^{-1}(\olom^c))$. Therefore, using that $\bar{\rho}_t$ is concentrated on $\olom$, we conclude that $\sigma(A)=0$, showing that $\sigma$ is concentrated on $\Gamma_{\bar{v}}(\olom)$. 
Finally, \eqref{representation:2} implies \eqref{representation:1} since $\bar{\rho}_t$ is supported in $\olom$ and it coincides with $\rho_t$ in $\olom$. Also $\Gamma_{\bar{v}}(\olom)=\Gamma_{v}(\olom)$ by definition of $\bar{v}$, thus concluding the proof.
\end{proof}

\subsection{Proof of Theorem \ref{thm:extremal}} \label{subsection:proof}

Let $\alpha \geq 0$, $\beta >0$. We divide the proof into two parts.

\smallskip

\noindent\textit{Part 1:} $\{ (0,0)\} \cup \points \subset \ext(C_{\alpha,\beta})$.
\smallskip

\noindent We start by showing that $\{(0,0)\}\cup \points \subset C_{\alpha,\beta}$. The fact that $(0,0) \in C_{\alpha,\beta}$ follows immediately, since $(0,0)$ solves the continuity equation and $J_{\alpha,\beta}(0,0)=0$ (by Lemma \ref{lem:prop B}). Consider now $(\rho,m) \in \points$.
Notice that $(\rho,m) \in \mathcal{C}_{\alpha,\beta}$ satisfies the continuity equation in the sense of \eqref{cont weak}: indeed for every $\f \in C^1_c ((0,1) \times \olom)$ we have
\begin{equation} \label{lem:inclusion:13}
\begin{aligned}
\int_{(0,1) \times \olom}  \de_t \f \, d \rho +  \nabla \f \cdot dm & = a_{\gamma} \int_0^1  \de_t \f (t,\gamma(t)) + \nabla \f (t,\gamma(t)) \cdot \dot{\gamma}(t) \, dt 
  \\
& = a_{\gamma} \int_0^1 \frac{d}{dt} \f (t,\gamma(t)) \, dt =a_{\gamma} ( \f (1,\gamma(1)) - \f(0,\gamma(0)) ) = 0
\end{aligned}
\end{equation}
since $\f$ is compactly supported in $(0,1) \times \olom$. 
Moreover, thanks to the fact that $\rho \geq 0$ and $m=\dot{\gamma} \rho$, we can invoke \eqref{formula B} to obtain
\begin{equation} \label{inclusion:one}
J_{\alpha,\beta}(\rho,m)  
= a_{\gamma} \left( \frac{\beta}{2} \int_0^1  |\dot{\gamma}(t)|^2  \, dt + \alpha\right) = 1\,,
\end{equation}
proving that $(\rho,m) \in C_{\alpha,\beta}$.

We now want to show that any $(\rho,m) \in \{(0,0)\} \cup \points$ is an extremal point for $C_{\alpha,\beta}$. Hence assume that $(\rho^1,m^1), (\rho^2,m^2) \in C_{\alpha,\beta}$ are such that 
\begin{equation}\label{lem:inclusion:1}
	(\rho,m)=\lambda (\rho^1,m^1) + (1-\lambda) (\rho^2,m^2)
\end{equation}
for some $\lambda \in (0,1)$. We need to show that $(\rho,m)=(\rho^1,m^1) = (\rho^2,m^2)$.
Set $j \in \{1,2\}$.
Since $(\rho^j,m^j)$ is such that $J_{\alpha,\beta}(\rho^j,m^j) \leq 1$, from $ii)$ in Lemma \ref{lem:prop B} we have that $\rho^j \geq 0$ and $m^j=v^j\, \rho^j $ for some Borel field $v^j \colon X \to \R^d$. In particular, if $(\rho,m)=(0,0)$, \eqref{lem:inclusion:1} forces $(\rho^j,m^j)=0$, hence showing that $(0,0)$ is an extremal point of $C_{\alpha,\beta}$. 

Let us now consider the case $(\rho,m) \in \points$.
By \eqref{inclusion:one} we have $J_{\alpha,\beta}(\rho,m)=1$. 
From \eqref{lem:inclusion:1}, convexity of $J_{\alpha,\beta}$, and the fact that $J_{\alpha,\beta}(\rho^j,m^j) \leq 1$, $\lambda \in (0,1)$, we conclude
\begin{equation} \label{lem:inclusion:11} 
J_{\alpha,\beta}(\rho^j,m^j)=1\,.
\end{equation}
Since $(\rho^j,m^j)$ solves the continuity equation, $\rho^j \geq 0$ and $J_{\alpha,\beta}(\rho^j,m^j) = 1$, from Lemma \ref{lem:prop cont} we deduce that $\rho^j=dt \otimes \rho_t^j$ for some narrowly continuous curve $t \mapsto \rho^j_t \in \M^+(\olom)$, with $\rho^j_t(\olom)$ constant in time. We define $a_j:=\rho^j_0(\olom)$ and notice that $a_j >0$: Indeed, $a_j=0$ would imply $\rho^j=0$, yielding $J_{\alpha,\beta}(\rho^j,m^j)=J_{\alpha,\beta}(0,0)=0$. This would contradict \eqref{lem:inclusion:11}.    
Now, from condition \eqref{lem:inclusion:1} 
and uniqueness of the disintegration we deduce
\begin{equation} \label{lem:inclusion:6}
a_{\gamma}\,\delta_{\gamma(t)} = \lambda  \rho_t^1 + (1-\lambda) \rho_t^2  \quad \text{ for every } t  \in [0,1] \,.
\end{equation}
Since $a_j >0$ (and hence $\rho^j_t \neq 0$), the above equality implies that $\supp \rho_t^j = \{\gamma(t)\}$, i.e.,
\begin{equation} \label{lem:inclusion:12}
\rho_t^j =a_j \,\delta_{\gamma(t)} \quad \text{ for every } t  \in [0,1] \,.
\end{equation}
We now show that $v^j=\dot{\gamma}$ on $\supp \rho = \gr(\gamma):=\{ (t,\gamma(t)) \colon t \in (0,1)\}$, that is
\begin{equation} \label{lem:inclusion:field}
	v^j(t,\gamma(t))= \dot\gamma(t) \quad \text{ for a.e. } \,\, t \in (0,1)\,.
\end{equation}
By assumption, $\de_t \rho^j + \div m^j=0$ in the sense of \eqref{cont weak}. %
Therefore, recalling \eqref{lem:inclusion:12} and the fact that $a_j >0$, we get that for each $\f \in C^1_c((0,1) \times \olom)$,
\begin{equation} \label{lem:inclusion:3}
\begin{aligned}
0 & = \int_0^1 \de_t \f (t, \gamma (t)) + \nabla \f (t,\gamma(t)) \cdot v^j(t, \gamma(t)) \, dt  \\
& =  \int_0^1 \de_t \f (t, \gamma (t)) + \nabla \f (t,\gamma(t)) \cdot \dot{\gamma}(t) \, dt \, + 
 \int_0^1 \nabla \f (t,\gamma(t)) \cdot ( v^j(t, \gamma(t)) - \dot{\gamma}(t) ) \, dt \\
& = \int_0^1 \nabla \f (t,\gamma(t)) \cdot ( v^j(t, \gamma(t)) - \dot{\gamma}(t) ) \, dt\,,
\end{aligned}
\end{equation}
where the last equality follows from \eqref{lem:inclusion:13}, since $a_\gamma >0$. Let $\psi \in C^1_c ((0,1))$ and define $\f(t,x):=x_i \psi (t) $, where $x=(x_1,\dots,x_d)$, so that $\f$ is a test function for \eqref{lem:inclusion:3}. By plugging $\f$ into \eqref{lem:inclusion:3} we obtain
\[
\int_0^1  \psi (t) ( v^j_i(t, \gamma(t)) - \dot{\gamma}_i(t) ) \, dt = 0 \qquad \text{ for all} \,\,\, i \in \{1,\ldots,d\}, \,\,j\in \{1,2\}
\]
where $v_i^j$ and $\dot{\gamma}_i$ are the $i$-th component of $v^j$ and $\dot{\gamma}$, respectively. 
This implies that $v^j(t,\gamma(t))=\dot{\gamma}(t)$ for a.e.~$t \in (0,1)$, that is, $v^j=\dot{\gamma}$ a.e.~on $\gr(\gamma)$. With this at hand, by means of \eqref{formula B} we can see that 
$J_{\alpha,\beta}(\rho^j,m^j)=a_j /a_\gamma$. Since \eqref{lem:inclusion:11} holds, we obtain $a_j= a_\gamma$, thus proving $(\rho,m)=(\rho^j,m^j)$ and hence extremality for $(\rho,m)$ in $C_{\alpha,\beta}$.

\medskip
\noindent\textit{Part 2:} $\ext(C_{\alpha,\beta}) \subset \{(0,0)\} \cup \points$.
\smallskip

\noindent Let $(\rho,m) \in C_{\alpha,\beta}$ be an extremal point. In particular, $J_{\alpha,\beta}(\rho,m) \leq 1$ so that by Lemma \ref{lem:prop B} $ii)$, we obtain $\rho \geq 0$ and $m=v \rho $ for some Borel field $v \colon X \to \R^d$. Notice that by extremality of $(\rho,m)$ and one-homogeneity of $J_{\alpha,\beta}$ we immediately infer
that either
$J_{\alpha,\beta}(\rho,m)=0$ or $J_{\alpha,\beta}(\rho,m)=1$. 
If $J_{\alpha,\beta}(\rho,m)=0$, by decomposing $(\rho,m)$ as 
\[
(\rho,m) = \frac{1}{2}(2\rho,2m) + \frac{1}{2}(0,0)
\] 
and using the extremality of $(\rho,m)$ together with the one-homogeneity of $J_{\alpha,\beta}$ we deduce that $(\rho,m) = (0,0)$. 
Thus, we consider the case 
\begin{equation}\label{thm:extremal:3}
J_{\alpha,\beta}(\rho,m)=1\,.
\end{equation}
Since by definition, $(\rho,m)$ solves the continuity equation in the sense of \eqref{cont weak} and $J_{\alpha,\beta}(\rho,m) = 1$, we can apply Lemma \ref{lem:prop cont} to obtain that $\rho=a \, dt \otimes \rho_t$ for some narrowly continuous curve $t \mapsto \rho_t \in \prob (\olom)$, where $a:=\rho(X)>0$. 

\medskip
\noindent\textit{Claim:} $\supp \rho_t$ is a singleton for each $t \in [0,1]$. 
\medskip

\noindent\textit{Proof of Claim:} The hypotheses of Theorem \ref{thm:lifting} are satisfied, therefore there exists a measure $\sigma \in \prob(\Gamma)$ concentrated on $\Gamma_v (\olom)$ and such that $\rho_t = (e_t)_\# \sigma$ for every $t \in [0,1]$. 
Assume by contradiction that there exists a time $\hat{t} \in [0,1]$ such that $\supp \rho_{\hat{t}}$ is not a singleton.
 Therefore, we can find a Borel set $E \subset \olom$ such that
\begin{equation} \label{thm:extremal:2bis}
0<\rho_{\hat{t}} (E)<1, \,\,\,\,\,\, 0< \rho_{\hat{t}} (\olom \smallsetminus E )<1 \,.
\end{equation}
Define the Borel set
\[
A:=\{ \gamma \in \Gamma \, \colon \, \gamma(\hat{t}) \in E \} = e^{-1}_{\hat{t}} (E) \,.
\]
By the relation $\rho_t = (e_t)_\# \sigma$ and definition of $A$ we obtain
$\rho_{\hat t}(E) =\sigma(A)$.
Therefore, from \eqref{thm:extremal:2bis}
\begin{equation}\label{thm:extremal:2} 
 0<{\sigma} (A)<1,\,\, \, \,\,\, 0 < \sigma (\Gamma \smallsetminus A )<1 \,.
\end{equation}
Define 
\[
\begin{aligned}
\lambda_1 & :=a \left( \frac{\beta}{2} \int_0^1 \int_{A} |\dot\gamma(t)|^2 \, d\sigma(\gamma) \, dt + \alpha \, \sigma(A) \right)\, ,\\  
\lambda_2 & :=a \left( \frac{\beta}{2} \int_0^1 \int_{A^c} |\dot\gamma(t)|^2 \, d\sigma(\gamma)  \, dt + \alpha \, \sigma(A^c)\right) \,,
\end{aligned}
\]
where $A^c:=\Gamma \smallsetminus A$. 
Note that $\lambda_1,\lambda_2$ are well defined (possibly being equal to $+\infty$) as the map
\begin{equation} \label{eq:mapL}
L(\gamma):=\int_0^1 |\dot{\gamma}(t)|^2\, dt  \,\,\,\, \text{ if } \,\, \gamma \in \AC^2([0,1];\R^d) \,, \,\,\,\, L(\gamma):=+\infty \,\,\,\, \text{ otherwise}, 
\end{equation}
is lower semicontinuous on $\Gamma$, and hence measurable.	
Notice that
\begin{equation} \label{thm:extremal:4}
\lambda_1 + \lambda_2 = a \left(\frac{\beta}{2} \int_0^1 \int_{\Gamma} |\dot\gamma(t)|^2 \, d\sigma(\gamma) \, dt + \alpha \right) = 
a \left( \frac{\beta}{2} \int_0^1 \int_{\Gamma} |v(t,\gamma(t))|^2 \, d\sigma(\gamma) \, dt + \alpha \right)\,,
\end{equation}
because $\sigma$ is concentrated on $\Gamma_v (\olom)$. 
Since $v(t,\cdot)$ belongs to $L^2_{\rho_t}(\olom;\R^d)$ for a.e.~$t \in (0,1)$, by \cite[Theorem 3.6.1]{bogachev} we obtain that the representation formula \eqref{representation:1} holds for $\f(x):=v(t,x)$ and a.e.~$t \in (0,1)$, that is, 
\begin{equation} \label{thm:extremal:5}
 \int_{\Gamma} |v(t,\gamma(t))|^2 \, d\sigma(\gamma) = 
 \int_{\olom} |v(t,x)|^2 \, d\rho_t(x) \,\,\, \text{ for a.e.} \,\,
 t \in (0,1)\,.
\end{equation}
Therefore, from \eqref{formula B}, \eqref{thm:extremal:3}, \eqref{thm:extremal:4} and \eqref{thm:extremal:5} we deduce $\lambda_1+\lambda_2=J_{\alpha,\beta}(\rho,m)=1$.

We now proceed with the proof of the claim separately for the cases $\alpha > 0$ and $\alpha = 0$.
Suppose first $\alpha> 0$. Notice that $\lambda_1,\lambda_2 > 0$ thanks to \eqref{thm:extremal:2} and the fact that $a>0$.
Decompose $(\rho,m)$ as 
\begin{equation} \label{thm:extremal:6}
(\rho,m)= \lambda_1 (\rho^1, m^1) + \lambda_2 (\rho^2, m^2 ) \,,
\end{equation}
where we defined 
\begin{equation} \label{thm:extremal:7}
\rho^j:= \frac{a}{\lambda_j} \, dt \otimes (e_t)_\# \sigma_j \,, \quad m^j:=\rho^j v \,,
\end{equation}
for $j=1,2$, with $\sigma_1 := \sigma \zak A$ and $\sigma_2 := \sigma \zak A^c$. Notice that $\rho^j \in \M^+(X)$, since $\sigma$ is a positive measure concentrated on $\Gamma_v (\olom)$, and $a,\lambda_j>0$. We now claim that $(\rho^j,m^j) \in C_{\alpha,\beta}$. First, we prove that $\de_t \rho^j + \div m^j=0$ in the sense of \eqref{cont weak}. Let $j=1$ and fix $\f \in C^1_c((0,1)\times \olom)$. Since $v(t,\cdot)$ belongs to $L^2_{\rho_t}(\olom;\R^d)$ for a.e. $t \in (0,1)$, by \cite[Theorem 3.6.1]{bogachev},  \eqref{representation:1} and the definition of $\sigma_1$, we get
\begin{equation}\label{eq:continuityeqextremality1}
\begin{aligned}
\int_X \de_t \f \, d\rho^1 +  \nabla \f \cdot d m^1 =  
& \frac{a}{\lambda_1} \int_0^1 \int_{\olom}  \de_t \f (t,x) + \nabla \f (t,x) \cdot v (t,x) \, d \left( (e_t)_\# \sigma_1 \right)(x) \, dt \\
 = & \frac{a}{\lambda_1}\int_0^1 \int_{A}  \de_t \f (t,\gamma(t)) + \nabla \f (t,\gamma(t)) \cdot v(t,\gamma(t)) \, d \sigma (\gamma) \, dt\,.
\end{aligned}
\end{equation}
Now recall that $\sigma$ is concentrated on $\Gamma_v (\olom)$ and that $\f$ is compactly supported in time, so that
\begin{equation}\label{eq:continuityeqextremality2}
\begin{aligned}
\int_X \de_t \f \, d\rho^1 +  \nabla \f \cdot d m^1 
 = & \frac{a}{\lambda_1}\int_0^1 \int_{A}  \de_t \f (t,\gamma(t)) + \nabla \f (t,\gamma(t)) \cdot \dot{\gamma}(t) \, d \sigma (\gamma) \, dt\\
 = &\frac{a}{\lambda_1}  \int_A \left( \int_0^1 \frac{d}{dt} \f (t,\gamma(t)) \, dt \right) \, d\sigma(\gamma) = 0 \,.
\end{aligned}
\end{equation}
The calculation for $j=2$ is similar. Also, 
by definition of $(\rho^j,m^j)$ and of $\lambda_j$, one can perform similar calculations to the ones in \eqref{thm:extremal:4}, \eqref{thm:extremal:5}, and prove that $J_{\alpha,\beta}(\rho^j,m^j)=1$. Hence $(\rho^j,m^j) \in C_{\alpha,\beta}$. We now claim that $(\rho^1,m^1) \neq (\rho^2,m^2)$. Suppose by contradiction that $(\rho^1,m^1) = (\rho^2,m^2)$. Then in particular $\rho^1 = \rho^2$, so that by \eqref{thm:extremal:7} we get
\begin{equation}\label{diff}
\frac{(e_t)_\# \sigma_1}{\lambda_1} =\frac{ (e_t)_\# \sigma_2 }{\lambda_2}\quad \text{for a.e. } t\in (0,1)\,.
\end{equation}
As $(\rho^j,m^j)$ are solutions of the continuity equation and $J_{\alpha,\beta}(\rho^j,m^j)=1$, from Lemma \ref{lem:prop cont}  it follows that the maps $ t \mapsto (e_t)_\# \sigma_j$ are narrowly continuous. In particular, \eqref{diff} holds for each $t \in [0,1]$. However, by \eqref{thm:extremal:2} and by definition of $A$, $\sigma_1$, $\sigma_2$,  we have
\[
 [(e_{\hat t})_\# \sigma_1](E) =   \sigma(A)>0 \,, \quad  \quad  [(e_{\hat t})_\# \sigma_2](E) = \sigma(\emptyset) = 0 \,,
\]
which contradicts \eqref{diff}. Therefore $(\rho^1,m^1) \neq (\rho^2,m^2)$, which shows that the decomposition \eqref{thm:extremal:6} is non-trivial. This is a contradiction, since we are assuming that $(\rho,m)$ is an extremal point for $C_{\alpha,\beta}$. Thus the claim follows.

Suppose now that $\alpha = 0$ and define the set
\begin{equation*}
Z := \left\{\gamma \in \Gamma : \int_0^1 |\dot \gamma(t)|^2 \, dt  = 0\right\}\,.
\end{equation*}
Notice that $Z$ is measurable, due to the measurability of the map $L$ at \eqref{eq:mapL}. We claim that $\sigma(Z)=0$. In order to prove that, let $Z^c:= \Gamma \smallsetminus Z$ and define the measures $\sigma_Z :=\sigma \zak Z$, $\sigma_{Z^c} :=\sigma \zak Z^c$, so that $\sigma=\sigma_Z + \sigma_{Z^c}$. Recalling that $\rho_t=(e_t)_{\#}\sigma$ for all $t \in [0,1]$, we can decompose
\begin{equation}\label{eq:convexcomb}
(\rho,m) = \frac{1}{2}(\rho^1,m^1) + \frac{1}{2}(\rho^2,m^2) \,,
\end{equation} 
where
\begin{equation*}
\begin{aligned}
\rho^1 & := a \, dt \otimes (e_t)_\# \sigma_{Z^c} \,, \,\,\,\, & m^1:= v \rho^1\, ,\\   
\rho^2 & :=  a\, dt  \otimes  (e_t)_\# \sigma_{Z^c}  + 2a \, dt \otimes (e_t)_\# \sigma_{Z}\,, \,\,\,\, & m^2:= v \rho^2\,.
\end{aligned}
\end{equation*}
Let $j=1,2$. Notice that $\rho^j \in \mathcal{M}^+(X)$ since $\sigma$ is a positive measure concentrated on $\Gamma_v (\olom)$ and $a > 0$. Following similar computation as \eqref{eq:continuityeqextremality1}-\eqref{eq:continuityeqextremality2} we infer that $(\rho^j,m^j)$ solves the continuity equation in the sense of \eqref{cont weak}. Moreover, by definition of $Z$ and the fact that $\sigma$ is concentrated on $\Gamma_v (\olom)$ we obtain
\begin{equation} \label{eq:fubini}
\begin{aligned}
	\int_0^1 \int_{Z^c} |v(t,\gamma(t))|^2 \, d\sigma(\gamma)\,dt & =
	\int_0^1 \int_{Z^c} |\dot\gamma(t)|^2 \, d\sigma(\gamma)\,dt \\
	& = \int_{Z^c} \int_0^1 |\dot\gamma(t)|^2 \, dt \, d\sigma(\gamma)\\
	 & =  \int_{\Gamma} \int_0^1 |\dot\gamma(t)|^2 \, dt \, d\sigma(\gamma) =  \int_0^1\int_{\Gamma} |v(t,\gamma(t))|^2 \, d\sigma(\gamma) \, dt \,,  
\end{aligned}
\end{equation}
where we employed Fubini's Theorem, which holds thanks to the measurability of the map $L$ and the identity \eqref{thm:extremal:5}, the latter implying boundedness of the last term in \eqref{eq:fubini}. By \eqref{eq:fubini} and arguing as in \eqref{thm:extremal:4}-\eqref{thm:extremal:5}, it is immediate to check that $J_{0,\beta}(\rho^j,m^j)=J_{0,\beta}(\rho,m)$. Recalling \eqref{thm:extremal:3} we then obtain $(\rho^j,m^j) \in C_{0,\beta}$.
As $(\rho,m)$ is an extremal point of $C_{0,\beta}$, from \eqref{eq:convexcomb} we deduce that $(\rho^1,m^1) = (\rho^2,m^2)$ and thus $ dt \otimes (e_t)_\# \sigma_{Z} = 0$. In particular, there exists $\hat t \in [0,1]$ such that $(e_{\hat t})_\# \sigma_{Z} = 0$. Hence for every $E \subset \Gamma$ measurable, by the positivity of $\sigma$, we have $\sigma_{Z}(E) \leq  (e_{\hat t})_\# \sigma_{Z}(e_{\hat t}(E))=0$, implying that $\sigma_{Z} = 0$. By \eqref{thm:extremal:2} and the definition of $\lambda_1$, $\lambda_2$, we conclude that $\lambda_1,\lambda_2 > 0$. With this property established, the claim that $\supp \rho_t$ is a singleton for each $t \in [0,1]$ follows by repeating the same arguments of the case $\alpha>0$, employing the decomposition of $(\rho,m)$ as in \eqref{thm:extremal:6}.

\medskip

We have shown that for each $t \in [0,1]$,  $\supp \rho_t$ is a singleton. We now conclude the proof of Theorem \ref{thm:extremal}. Since $\rho_t \in \prob(\olom)$, the latter implies the existence of a curve $\gamma \colon [0,1] \to \olom$ such that $\rho_t=\delta_{\gamma(t)}$ for each $t \in [0,1]$. We will now prove that $\gamma \in \AC^2([0,1];\R^d)$. By narrow continuity of $t \mapsto \rho_t$, we have that the map $t \mapsto  \, \f(\gamma(t))$
is continuous for all $\f \in C(\olom)$. By testing against the coordinate functions $\f(x):=x_i$, we obtain continuity for $\gamma$. Consider now $\f(t,x):=\xi(t)\eta(x)$ with $\xi \in C^{\infty}_c((0,1))$, $\eta \in C^1(\olom)$. Notice that the scalar map $t \mapsto \eta(\gamma(t))$ is continuous. Moreover, by testing the continuity equation $\de_t \rho + \div ( v \rho)=0$ against $\f$ we get
\[
\int_0^1 \xi'(t) \, \eta(\gamma(t)) \, dt = - \int_0^1 \xi(t) \, \nabla \eta (\gamma(t)) \cdot v(t,\gamma(t)) \, dt  \,,
\]
which implies that the distributional derivative of the map $t \mapsto \eta(\gamma(t))$ is given by 
\[
t \mapsto \nabla \eta (\gamma(t)) \cdot v(t,\gamma(t)) \,.
\]
We now remark that the above map belongs to $L^2((0,1))$, since
\[
\int_0^1 |\nabla \eta (\gamma(t)) \cdot v(t,\gamma(t))|^2 \, dt \leq \norm{\nabla \eta}_{\infty} \int_0^1 |v(t,\gamma(t))|^2 \, dt \leq C\, \norm{\nabla \eta}_{\infty} J_{\alpha,\beta}(\rho,m) < + \infty\,.
\]
Therefore, $t \mapsto \eta(\gamma(t))$ belongs to $\AC^2([0,1])$ for every fixed $\eta \in C^1(\olom)$. By choosing $\eta(x):=x_i$, $i = 1 , \dots, d$, we conclude that $\gamma \in \AC^2([0,1];\R^d)$. Since $\de_t \rho+ \div (v\rho )=0$, we can repeat the 
same argument employed to prove \eqref{lem:inclusion:field}, and infer
\begin{equation}  \label{lem:inclusion:30}
v(t,\gamma(t)) = \dot{\gamma}(t) \,\, \text{ a.e. in } \,\, (0,1) \,.
\end{equation}
From \eqref{lem:inclusion:30} and the fact that $\rho=a \, dt \otimes \delta_{\gamma(t)}$, $m=v\rho$, we then conclude $m=\dot\gamma \rho$. As $a>0$, we can apply \eqref{formula B} to compute
\begin{equation}\label{eq:lastJ}
J_{\alpha,\beta}(\rho,m)=a \left( \frac{\beta}{2} \int_0^1 |\dot{\gamma}(t)|^2 \, dt + \alpha \right) \,.
\end{equation}
Recalling that $J_{\alpha,\beta}(\rho,m)=1$ (see \eqref{thm:extremal:3}), from \eqref{eq:lastJ} we conclude that $a=a_{\gamma}$ with $a_{\gamma}<+\infty$. Therefore $(\rho,m)$ belongs to $\points$ according to Definition \ref{def:char}, and the proof of Theorem \ref{thm:extremal} is concluded.

\section{Application to sparse representation for inverse problems with optimal transport regularization} \label{sec:application}

In this section we deal with the problem of reconstructing a family of time-dependent Radon measures given a finite number of observations. To be more specific, let $H$ be a finite dimensional Hilbert space and $A: \curves \rightarrow H$ be a linear continuous operator, where continuity is understood in the following sense: 
given a sequence $(t \mapsto \rho^n_t)$ in $\curves$, we require that
\begin{equation}\label{eq:topologycweak}
\rho_t^n  \weakstar \rho_t \quad \text{ weakly* in } \,\, \M(\olom)  \,\,\text{ for all }\,\, t\in [0,1] \,\, \text{ implies } \,\, A\rho^n \rightarrow A\rho \ \text{ in } H\,,
\end{equation}
where, with a little abuse of notation, we will denote by $\rho^n$ both the curve $t \mapsto \rho_t^n$, as well as the measure $\rho^n:=dt \otimes \rho^n_t$. 

For some given data $y \in H$, we aim to reconstruct a solution $\rho \in \curves$ to the dynamic inverse problem
\begin{equation} \label{appl:inverse}
A \rho = y \,. 
\end{equation}
For $\alpha >0$ and $\beta >0$ we regularize the above inverse problem by means of the energy $J_{\alpha,\beta}$ defined in \eqref{prel:reg}, following the approach in \cite{bf}. In practice, upon introducing the space 
\[
\widetilde{\mathcal{M}}:= \curves \times \mathcal{M}(X;\R^d)\,,
\]  
we consider the Tikhonov functional $G:\widetilde{\mathcal{M}}\rightarrow \R \cup \{+\infty\}$ defined as  
\begin{equation}\label{eq:inversefunctional}
G(\rho,m) =J_{\alpha,\beta}(\rho,m) + F(A\rho)\,, 
\end{equation}
where $F: H \rightarrow \R \cup\{+\infty\}$ is a given fidelity functional for the data $y$, which is assumed to be convex, lower semicontinuous and bounded from below. Additionally, we assume that $G$ is proper. 
We then replace \eqref{appl:inverse} by
\begin{equation}\label{eq:inverseproblem}
\min_{(\rho,m) \in \widetilde{\mathcal{M}} } \ G(\rho,m) \, .
\end{equation}

\begin{rem}
Two common choices for the fidelity term $F$ in the case $H=\R^k$ are, for example, 
\begin{itemize}
\item [i)] $F(x) = I_{\{y\}}(x)$ for a given $y \in \R^k$ that forces the constraint $A\rho = y$,
\item[ii)] $F(x) = \frac{1}{2}\|x-y\|_2^2$ that recovers a classical $l^2$ penalization.
\end{itemize}
\end{rem}

\begin{rem}\label{rem:existence}
Under the above assumptions on $A$ and $F$, problem \eqref{eq:inverseproblem} admits a solution. 
Indeed, since $G$ is proper, any minimizing sequence $\{(\rho^n,m^n)\}_n$ is such that $\{G(\rho^n,m^n)\}_n$ is bounded. As $F$ is bounded from below and $J_{\alpha,\beta} \geq 0$, we deduce that $\{J_{\alpha,\beta}(\rho^n,m^n)\}_n$ is bounded. Therefore, Lemma \ref{lem:prop J} implies that $(\rho^n,m^n)$ converges (up to subsequences) to some $(\rho,m) \in \widetilde{\M}$, in the sense of \eqref{topology}. By weak* lower semicontinuity of $J_{\alpha,\beta}$ in $\M$ (see Lemma \ref{lem:prop J}) and by \eqref{eq:topologycweak} together with the lower-semicontinuity of $F$, we infer that $(\rho,m)$ solves \eqref{eq:inverseproblem}. 
\end{rem}

It is well-known that the presence of a finite-dimensional constraint in an inverse problem, such as \eqref{appl:inverse}, promotes sparsity in the reconstruction. This observation has been recently made rigorous in \cite{bc} and \cite{chambolle}, where it has been shown that the atoms of a sparse minimizer are the extremal points of the ball of the regularizers. In Theorem \ref{thm:extremal}, we provided a characterization for the extremal points of the ball of $J_{\alpha,\beta}$. Therefore, specializing the above-mentioned results to our setting yields the following characterization theorem for sparse minimizers to \eqref{eq:inverseproblem}.

\begin{thm}\label{thm:sparserepresentation}
Let $\alpha,\beta>0$. There exists a minimizer $(\hat \rho, \hat m) \in \widetilde{\M}$ of \eqref{eq:inverseproblem} that can be represented as
\begin{equation}
(\hat \rho, \hat m) = \sum_{i=1}^p c_i \, (\rho^i,m^i)\,,
\end{equation}
where $p\leq \dim(H)$, $c_i >0$, $\sum_{i=1}^p c_i = J_{\alpha,\beta}(\hat \rho,\hat m)$, and
\[
\rho^i= a_{\gamma_i} \, dt \otimes  \delta_{\gamma_i(t)} \,, \,\,\, m^i = \dot{\gamma}_i \,  \rho^i\,,
\]
where $\gamma_i \in \AC^2([0,1];\R^d)$ with $\gamma(t) \in \olom$ for each $t \in [0,1]$, and $a_{\gamma_i}^{-1}:=\frac{\beta}{2} \int_0^1 |\dot\gamma_i|^2 \, dt + \alpha$.  
\end{thm}
In other words, the above theorem ensures the existence of a minimizer of \eqref{eq:inverseproblem} which is a finite linear combination of measures concentrated on the graphs of $\AC^2$-trajectories contained in $\olom$. In Section \ref{subsec:proof2} we give a proof of Theorem \ref{thm:sparserepresentation}, and we conclude the paper with Section \ref{subsec:dynamic}, where we apply the sparsity result of Theorem \ref{thm:sparserepresentation} to dynamic inverse problems with optimal transport regularization, following the approach of \cite{bf}.

\subsection{Proof of Theorem \ref{thm:sparserepresentation}} \label{subsec:proof2}
As already mentioned, the proof is an immediate consequence of Theorem \ref{thm:extremal} and a particular case of \cite[Corollary $2$]{chambolle} (see also \cite[Theorem $1$]{chambolle}). Before proceeding with the proof, for the reader's convenience, we recall Corollary $2$ from \cite{chambolle}. The definitions appearing in the statement below will be briefly recalled in the proof of Theorem \ref{thm:sparserepresentation}.

\begin{thm}[\cite{chambolle}]\label{thm:chambolle}
Let $\mathcal{U}$ be a locally convex space, $H$ be a finite-dimensional Hilbert space, $R:\mathcal{U} \rightarrow [-\infty,+\infty]$, $F: H\rightarrow  [-\infty,+\infty]$ be convex, and $A:\mathcal{U} \rightarrow H$ be linear. 
Consider the variational problem
\begin{equation}\label{eq:thmchambolle}
\inf_{u\in \mathcal{U}} \, R(u) + F(Au)\, .
\end{equation}
Suppose that the set of minimizers of \eqref{eq:thmchambolle}, denoted by $S$, is non-empty. Additionally, assume that there exists $\hat u \in \ext(S)$ such that the set 
\begin{equation}
\overline{C} = \{u\in \mathcal{U}: R(u) \leq R(\hat u)\}
\end{equation} 
is \emph{linearly closed}, the \emph{lineality space} of $\overline{C}$ is $\{(0,0)\}$
and $\inf_{u\in \mathcal{U}} R(u) < R(\hat u)$. Then, exactly one of the following conditions holds: 
\begin{itemize}
\item[i)] $\hat{u}$ is a convex combination of at most $\dim (H)$ extremal points of $\overline{C}$,  
\item[ii)] $\hat{u}$ is as a convex combination of at most  $\dim (H) -1$ points, which are either extremal points of $\overline{C}$, or belong to an extreme ray of $\overline{C}$.
\end{itemize}
\end{thm}

\begin{proof}[Proof of Theorem \ref{thm:sparserepresentation}]
We just need to verify that we can apply Theorem \ref{thm:chambolle} to the variational problem \eqref{eq:inverseproblem}. So, we choose $\mathcal{U}= \widetilde M$, $R=J_{\alpha,\beta}$ and $F$ and $A$ satisfying the assumptions stated above. Let $S$ be the set of solutions to \eqref{eq:inverseproblem}. 
 
First, notice that in Remark \ref{rem:existence} we have already shown that  the set of minimizers for \eqref{eq:inverseproblem} is non-empty, so that $S \neq \emptyset$. Moreover $S$ is compact with respect to the weak* topology. Indeed, given a sequence $(\rho^n,m^n)$ in $S$ we can use Lemma \ref{lem:prop J} to extract a subsequence (not relabelled) such that $(\rho^n,m^n) \weakstar (\rho,m)$ in $\mathcal{M}$ and $\rho_t^n \weakstar \rho_t$ in $\mathcal{M}(\olom)$ for every $t\in [0,1]$. Using the sequential lower semicontinuity of $J_{\alpha,\beta}$ with respect to weak* convergence combined with the continuity of $A$ (according to \eqref{eq:topologycweak}) and the lower semicontinuity of $F$, we obtain $(\rho,m) \in S$. We conclude that $S$ is sequentially weakly* compact and hence weakly* compact, thanks to the metrizability of the weak* convergence on bounded sets. Finally note that $S$ is convex thanks to the convexity of $F$ and $J_{\alpha,\beta}$ (Lemma \ref{lem:prop J}). By Krein--Milman's Theorem, we then infer the existence of a $(\hat{\rho},\hat m) \in \ext(S)$. 

The lineality space of $\overline{C}$ is defined as $\text{lin}(\overline{C}) = \text{rec}(\overline{C}) \cap (-\text{rec}(\overline{C}))$, where $\text{rec}(\overline{C})$ is the recession cone of $\overline{C}$ defined as the set of all $(\rho,m)\in \mathcal{U}$ such that $\overline C +\R_+(\rho,m)  \subset \overline C$. Hence, from the coercivity of $J_{\alpha,\beta}$ in Lemma \ref{lem:prop J} it is immediate to conclude that $\text{lin}(\overline{C}) = \{(0,0)\}$. 
Moreover, $\overline{C}$ is linearly closed if the intersection of $\overline{C}$ with every line is closed. It is easy to verify that, as $\overline{C}$ is weakly* closed (Remark \ref{rem:existence}), it is also linearly closed. Finally, the assumption  $\inf_{(\rho,m)\in \widetilde M} J_{\alpha,\beta}(\rho,m) < J_{\alpha,\beta}(\hat \rho, \hat m)$ is satisfied whenever $(\hat \rho, \hat m) \neq 0$, as in this case $J_{\alpha,\beta}( \hat \rho,  \hat m) >0$, while $\inf_{(\rho,m)\in \widetilde M} J_{\alpha,\beta}(\rho,m)=0$.
Hence, the hypotheses of Theorem \ref{thm:chambolle} for the functional \eqref{eq:inversefunctional} are verified. Notice also that $\overline C$ does not contain extreme rays. In order to prove that, we first recall that a ray of $\overline C$ is any set of the form $r_{p,v} = \{p+ tv : t >0\}$ for $p, v\in \overline C$, $v\neq 0$. An extreme ray of $\overline C$ is a ray $r_{p,v}$ such that for every segment intersecting $r_{p,v}$, the whole segment is contained in $r_{p,v}$.
Thanks to the coercivity of $J_{\alpha,\beta}$ in Lemma \ref{lem:prop J}, it is immediate to see that $\overline C$ contains no rays and thus no extreme rays.
Hence, from either of the conclusions $i),ii)$ in Theorem \ref{thm:chambolle}, we deduce that there exists a minimizer $(\hat \rho, \hat m) \in \mathcal{M}$ of \eqref{eq:inverseproblem} that can be represented as
\begin{equation}\label{eq:1representation}
(\hat \rho, \hat m) = \sum_{i=1}^p c_i (\rho_i,m_i)\, ,
\end{equation}
where $(\rho_i,m_i)\in \ext(C_{\alpha,\beta})$, $p\leq \dim(H)$, $c_i>0$ and $\sum_{i=1}^p c_i = J_{\alpha,\beta}(\hat \rho,\hat m)$.
We remark that if $(\hat \rho, \hat m) = 0$, the assumption $\inf_{(\rho,m)\in \widetilde M} J_{\alpha,\beta}(\rho,m) < J_{\alpha,\beta}(\hat \rho, \hat m)$ in Theorem \ref{thm:chambolle} is not satisfied, but the representation \eqref{eq:1representation} holds trivially.
Using the characterization of extremal points in Theorem \ref{thm:extremal} and \eqref{eq:1representation}, we obtain an explicit sparse representation for solutions of \eqref{eq:inverseproblem} and the proof is achieved.
\end{proof}

\subsection{Dynamic inverse problems} \label{subsec:dynamic}

Theorem \ref{thm:sparserepresentation} provides a representation formula for sparse solutions of \eqref{eq:inverseproblem} that holds for every $A$ and $F$ satisfying the above-stated hypotheses.
A relevant choice for $A$ and $F$ is proposed in \cite{bf} as a model for dynamic inverse problems: in particular, the authors apply their framework to variational reconstruction in undersampled dynamic MRI. In what follows we make an explicit choice of $F$ and $A$ in order to apply Theorem \ref{thm:sparserepresentation} to a special case of the framework in \cite{bf}, namely the case of discrete time sampling, and finite-dimensionality of the data for each sampled time. %

To be more specific, consider a discretization of the interval $[0,1]$ in $N$ points $t_1<t_2<\ldots<t_N$ and assume that we want to reconstruct an element of $\curves$, by only making observations at the time instants $t_1,\ldots,t_N$. To this aim, let $H_{t_i}$ be a family of finite-dimensional Hilbert spaces and introduce the product space $\mathcal{H}:=\bigtimes_{i=1}^N H_{t_i}$, normed by $\|y\|_{\mathcal{H}}^2 := \sum_{i=1}^N \|y_i\|^2_{H_{t_i}}$. Let
$A_{t_i} : \M(\olom) \to H_{t_i}$ be linear operators, which are assumed to be weak*  continuous for each $i=1,\ldots,N$.
For a given observation $(y_{t_1},\ldots,y_{t_N}) \in  \mathcal{H}$, consider the problem of finding $\rho \in \curves$ such that
\[
A_{t_i} \rho_{t_i} = y_{t_i} \,\, \text{ for each } \,\,
i = 1 ,\ldots,N\,.
\] 
Following \cite{bf}, we regularize the above problem by
\begin{equation}\label{eq:mri}
\min_{(\rho,m)\in \widetilde{\mathcal{M}}} \, J_{\alpha,\beta}(\rho,m) + \frac{1}{2}\sum_{i=1}^N \|A_{t_i}\rho_{t_i} - y_{t_i}\|^2_{H_{t_i}} \,.
\end{equation}
In order to recast the above problem into the form \eqref{eq:inverseproblem}, let $A : \curves \rightarrow \mathcal{H}$ be the linear operator defined by 
\[
A\rho := (A_{t_1}\rho_{t_1}, \ldots,A_{t_N}\rho_{t_N})\,.
\] 
Notice that $A$ is continuous in the sense of \eqref{eq:topologycweak}, thanks to the assumptions on $A_{t_i}$. We can then equivalently rewrite \eqref{eq:mri} as 
\begin{equation}\label{eq:mri2}
\min_{(\rho,m)\in \widetilde{\mathcal{M}}} \, J_{\alpha,\beta}(\rho,m) + \frac{1}{2}\|A\rho - y\|^2_{\mathcal{H}} \,.
\end{equation}
 In this way, we recover a problem of the type of \eqref{eq:inverseproblem}, where $F(x) := \frac{1}{2}\|x - y\|^2_{\mathcal{H}}$. Notice that $F$ is convex, lower semicontinuous and bounded from below. Moreover, the functional in  \eqref{eq:mri2} is proper, since $J_{\alpha,\beta}(0,0)=0$. Hence, we can apply Theorem \ref{thm:sparserepresentation} to conclude the following result.

 \begin{cor} \label{corollary}
Let $\alpha,\beta>0$. 
 There exists a minimizer $(\hat \rho, \hat m) \in \widetilde{\mathcal{M}}$ of \eqref{eq:mri} that can be represented as
\begin{equation}\label{eq:final}
(\hat \rho, \hat m) =  \sum_{i=1}^p c_i \, (\rho^i,m^i)\,,
\end{equation}
where $p\leq \dim(\mathcal{H}) = \sum_{i=1}^N \dim(H_i)$, $c_i >0$, $\sum_{i=1}^p c_i = J_{\alpha,\beta}(\hat \rho,\hat m)$, and
\[
\rho^i= a_{\gamma_i} \, dt \otimes  \delta_{\gamma_i(t)} \,, \,\,\, m^i = \dot{\gamma}_i \,  \rho^i\,,
\]
where $\gamma_i \in \AC^2([0,1];\R^d)$ with $\gamma(t) \in \olom$ for each $t \in [0,1]$, and $a_{\gamma_i}^{-1}:=\frac{\beta}{2} \int_0^1 |\dot\gamma_i|^2 \, dt + \alpha$.  
 \end{cor}

\begin{rem}
The upper bound $p \leq \sum_{i=1}^N \dim (H_i)$ in the representation formula \eqref{eq:final} might not be optimal. However, a careful analysis of the faces of the ball of the Benamou--Brenier energy, possibly under additional assumptions on the operator $A$ and fidelity term $F$, could be needed to substantiate such conjecture. We leave this question open for future research. 
\end{rem}

\section*{Acknowledgements}
We thank the referee for the useful suggestions provided, particularly for encouraging us to include the case $\alpha=0$ in the characterization of Theorem \ref{thm:extremal}, which was previously missing. 
KB and SF gratefully acknowledge support by the Christian Doppler Research Association (CDG) and Austrian Science Fund (FWF) through the Partnership in Research
project PIR-27 ``Mathematical methods for motion-aware medical imaging'' and project P 29192 ``Regularization graphs for variational imaging''.
MC is supported by the Royal Society (Newton International Fellowship NIF\textbackslash R1\textbackslash 192048 Minimal partitions as a robustness boost for neural network classifiers). The Institute of Mathematics and Scientific Computing, to which KB, SF and FR are affiliated, is a member of NAWI Graz (\texttt{http://www.nawigraz.at/en/}). The authors  KB, SF and FR are further members of/associated with BioTechMed Graz (\texttt{https://biotechmedgraz.at/en/}).

	\bibliography{bibliography.bib}
	\bibliographystyle{my_plain.bst}
\end{document}